\documentclass[12pt,notitlepage,a4paper,reqno,ps]{amsart}
\usepackage{amsfonts}
\usepackage{eurosym}
\usepackage{amssymb}
\usepackage{amstext}
\usepackage{xcolor}
\usepackage{ulem}

\definecolor{lightgray}{rgb}{0.83, 0.83, 0.83}

\newcommand{\N}{\mathbb{N}}

\newcommand{\medint}{-\kern  -,375cm\int}

\newenvironment{michelarev}{\color{red}}{\color{black}}
\newcommand{\bmicr}{\begin{michelarev}}
\newcommand{\emicr}{\end{michelarev}}
\theoremstyle{plain}
\newtheorem{theorem}{Theorem}[section]

\newtheorem{lemma}[theorem]{Lemma}
\newtheorem{proposition}[theorem]{Proposition}
\theoremstyle{definition}

\theoremstyle{remark}

\newtheorem{example}[theorem]{Example}

\def\R{\mathbb{R}}

\numberwithin{equation}{section} \makeatletter

\renewcommand{\p@enumi}{\thesection.}
\makeatother  \allowdisplaybreaks
\email{michela.eleuteri@unimore.it}
 \email{stefania.perrotta@unimore.it}
 \email{giulia.treu@unipd.it}

\begin{document}
\title[Lipschitz continuity for integrals under unified general growth]{Local Lipschitz continuity for energy integrals under unified general growth of fast and slow type}
\author[M. Eleuteri  -- S. Perrotta -- G. Treu]{Michela
Eleuteri --  Stefania Perrotta -- Giulia Treu}
\address{Dipartimento di Scienze Fisiche, Informatiche e Matematiche,
Universit\`a degli Studi di Modena e Reggio Emilia, via Campi 213/b, 41125 -
Modena, Italy}
\address{Dipartimento di Matematica `Tullio Levi-Civita', Universit\`a di Padova, Via Trieste 63, 35121 - Padova, Italy}
\thanks{The authors are members of GNAMPA (Gruppo Nazionale per l'Analisi
Matematica, la Probabilit\`a e le loro Applicazioni) of INdAM (Istituto Nazionale di Alta Matematica). M. Eleuteri has been partially supported through the INdAM-GNAMPA Project 2026 (CUP E53C25002010001) ``Esistenza e regolarità per soluzioni di equazioni ellittiche e paraboliche anisotrope''. G. Treu has been partially supported through the project 2025DM1DOR25-00669 UniPd}

\begin{abstract}
We consider a class of anisotropic energy integrals including fast and slow growth under a unified approach, proving that the
local minimizers are locally Lipschitz continuous.
\end{abstract}

\maketitle


\section{Introduction and statement of the main results\label{Section Prologue}}
A fundamental problem in the Calculus of Variations concerns the regularity properties of minimizers of integral functionals
\begin{equation}
F(u)=\int_{\Omega }f(Du(x)) \,dx \,.
\label{energy integral}
\end{equation}
While the classical theory mainly deals with standard polynomial growth conditions, many models arising in nonlinear analysis and continuum mechanics require substantially more general growth behaviors.
\\
A major breakthrough in the regularity theory of variational integrals was the introduction of {\it nonstandard growth conditions.} The regularity theory for nondegenerate functionals started with the seminal works of Paolo Marcellini \cite{mar89, mar91}, while the study of strongly degenerate functionals, including orthotropic functionals, was developed primarily in the Russian literature, see, for instance, \cite{UU}. 

Among nonstandard growth conditions, the $(p,q)-$growth framework has played a particularly prominent role, becoming a cornerstone of the modern theory. Indeed, it has subsequently led to essentially sharp conditions for the regularity of minimizers \cite{ELM04}: besides the already mentioned paper we further refer, without claiming to be exhaustive, to some Lipschitz regularity results \cite{BM20, BS20, CM15, CMMP23, DFM20, DFM21, DFM23,  EMM1, EMM3, EN23, ELM02, GMP24, Mar20, Mar21, MM03}, including the case of orthotropic structure with nonstandard growth (see \cite{BB20, BBL24, BBLV18, LOS26} and the references therein).
\\
In \cite{mar96} Marcellini drew attention toward genuinely nonstandard growth behaviors, by considering scalar functionals, depending only on the gradient, with fast growth regimes, i.e., superquadric ones reaching up to super-exponential growth; see also the recent contribution \cite{MNP26} where the non-autonomous counterpart of this result is established, although only at the level of the a priori estimate. At the opposite extreme, functionals with at most subquadratic growth, also referred to as 'slow growth',
 have been considered in \cite{EMMP}, also for non-autonomous integrands depending on lower-order terms \cite{EPT25}. These situations require refined techniques, since many standard arguments based on uniform polynomial ellipticity are no longer directly applicable.

The main purpose of this paper is to introduce a unified approach including, as particular cases, energies with both rapidly increasing and slowly increasing growth.\
Our main result applies to growth rates that are just above linear, such as \[
|\xi|\log(1+(\log(\cdots \log(1+|\xi|)\cdots)), 
\]
as well as those that grow faster than exponentially, such as 
\[
(\exp(\cdots(\exp(\exp(|\xi|^{p_0})^{p_1})^{p_2}\cdots)^{p_k},\]
with $p_i>0$, $i=1,\dots k$.

Previous attempts to develope a unified approach include the work of \cite{MP06}, where functionals exhibiting both slow and faster-than-exponential growth are treated, although the Lagrangian depends only on the modulus of the gradient; see also the paper \cite{DM20} for the corresponding a priori estimate in the non-autonomous setting.

\bigskip

\textbf{Regularity and growth assumptions.}  In the sequel we assume the following regularity and growth assumptions on $f$, {\tt (RGA)} for short:
\medskip

\noindent let $f:\mathbb{R}^{n}\rightarrow \lbrack 0\,,+\infty )$ be  convex, with $f\in\mathcal{C}(\mathbb{R}^{n})\cap \mathcal{C}^{2}(\mathbb{R%
}^{n}\backslash B_{t_{0}}(0))$ for some $t_{0}\geq 0$. Assume that there exist two continuous functions  $g_{1},g_{2}:[t_{0}\,,+\infty )\rightarrow (0\,,+\infty )$ and positive constants $C_{1}$, $C_{2}$, $\delta$ and $\alpha$  such that 
\begin{equation}
\left\{ 
\begin{split}
& g_{1}\left( \left\vert \xi \right\vert \right) \left\vert \lambda
\right\vert ^{2}\leq \sum_{i,j=1}^{n}f_{\xi _{i}\xi _{j}}\left( \xi \right)
\lambda _{i}\lambda _{j}\leq g_{2}\left( \left\vert \xi \right\vert \right)
\left\vert \lambda \right\vert ^{2}\,,\;\;\;\forall \;\lambda ,\xi \in 
\mathbb{R}^{n},\;\left\vert \xi \right\vert {\geq }t_{0} \\
&
t\mapsto t^{2} g_{1}(t) \,\,\, \text{is
increasing}\\
& \left(1+\int_{t_0}^t s\sqrt{g_1(s)}\,ds \right)^\delta\geq C_1\, t^{4} g_2(t),
\;\,\,\,\forall \text{$t\geq t_{0}$} \\
& |\xi |^{4} g_{2}(\left\vert \xi \right\vert )\leq C_{2}\,[1+f(\xi
)]^{\alpha },\;\;\;\forall \text{ $\xi \in \mathbb{R}^{n}$%
},\;\left\vert \xi \right\vert \text{$\geq t_{0}$,} 
\end{split}%
\right.  \label{H}
\end{equation}%
under the further assumption that
\begin{equation}
\label{Hpalphadelta}
2<\delta <2^*\qquad \qquad
1 < \alpha < \frac{2}{n} \frac{\delta}{\delta -2},
\end{equation}
where $2^* = \frac{2n}{n-2}$ if $n \ge 3$ while in the case $n=2$, it must be replaced with any fixed positive number greater than $2$.

\bigskip

We say
that $u\in W_{loc}^{1,1}(\Omega )$ is a \textit{local minimizer} of the
integral functional $F$ in \eqref{energy integral} if $f\left( Du\right) \in
L_{loc}^{1}(\Omega )$ and 
\begin{equation*}
\int_{\Omega ^{\prime }}f(Du)\,dx\leq \int_{\Omega ^{\prime }}f(Du+D\varphi
)\,dx
\end{equation*}%
for every open set $\Omega ^{\prime }$, such that $\overline{\Omega ^{\prime }}\subset
\Omega $, and for every $\varphi \in W_{0}^{1,1}(\Omega ^{\prime })$. 

\medskip

The first result of this paper is the following a priori estimate for the local Lipschitz constant of a locally Lipschitz minimizer.

\begin{theorem}[A priori estimate]
\label{step 1} Assume that $f$ satisfies the regularity and growth assumptions {\tt (RGA).} In addition, assume that $f(\xi) $ is of class $\mathcal{C}^{2}(\mathbb{R}^{n})$ and for every $
\Lambda>0$ there exists a positive constant $\ell =\ell (\Lambda)$ such that 
\begin{equation}
\ell \,|\lambda |^{2}\leq \,\sum_{i,j=1}^{n}f_{\xi _{i}\xi _{j}}(\xi
)\,\lambda _{i}\,\lambda _{j}\qquad \forall \lambda ,\xi \in \mathbb{R}%
^{n},\,|\xi |\le \, \Lambda.  \label{supp1}
\end{equation}%

If $u\in W_{\mathrm{loc}}^{1,\infty }(\Omega )$ is a local minimizer of %
\eqref{energy integral}, then, for every $%
0<\rho <R$, $\overline{B}_{R}\subset \Omega $, there exists a positive constant $%
C $ depending on $\rho $, $R$, $C_{1}$, $C_{2}$, $\alpha $, $\delta$, $n$, $g_{1}(t_0)$, but independent of $\ell$, and two positive exponents $\theta_1$, $\theta_2$ depending only on $\delta$, $\alpha$ and $n$ such that 
\begin{equation}
\Vert Du\Vert _{L^{\infty }(B_{\rho }\,;\mathbb{R}^{n})}\leq \,C\,\left\{ 
\frac{1}{(R-\rho )^{\theta_1}}\int_{B_{R}}(1+f(Du))\,dx\right\} ^{\theta_2}.
\label{infinity nu}
\end{equation}%
\end{theorem}

The main
result of this paper on the local Lipschitz regularity of minimizers for \textit{general growth conditions}, i.e. under the ellipticity condition 
\eqref{H}$_1$ with $g_{1}$
and $g_{2}$ general functions, can be stated as follows.
\begin{theorem}[Lipschitz regularity]
\label{superlinear} 
Assume that $f$ satisfies the regularity and growth assumptions {\tt (RGA).}
Suppose in addition that
\begin{align}
  & \label{fcoercitiva} \lim_{|\xi|\to\infty}\frac{f(\xi)} {|\xi|} =+\infty ;\\
  & \label{ipotesisug2}
   \liminf_{t\to+\infty}\left(t\,g_2(t)\right)\in(0\,,+\infty].
\end{align}
Then any local minimizer $u\in W_{\mathrm{loc}}^{1,1}(\Omega )$ of 
\eqref{energy
integral} is locally Lipschitz continuous in $\Omega $ and, for every $%
0<\rho <R$, $\overline{B}_{R}\subset \Omega $, there exists a positive constant $%
C $, depending on $\rho $, $R$, $C_{1}$, $C_{2}$, $\alpha $, $\delta$, $n$, $g_{1}(t_{0})$, and two exponents $\theta_1$, $\theta_2$, depending on $\alpha$ and $\delta$, such that 
\begin{displaymath}
\Vert Du\Vert _{L^{\infty }(B_{\rho }\,;\mathbb{R}^{n})}\leq 
C\left\{
\frac{1}{(R-\rho)^{\theta_1}}\int_{B_{R}}(1+f(Du))\,dx\right \} ^{\theta_2}.
\end{displaymath}%
\end{theorem}

Assumptions \eqref{fcoercitiva} and \eqref{ipotesisug2} are purely technical. Their role is to guarantee that the functional can be approximated by a sequence of functionals for which the required a priori estimate holds, and to justify passing to the limit in the minimality condition (see the proof of the theorem). It is worth emphasizing, however, that these assumptions are not needed in the derivation of the a priori estimate itself (see Example \ref{linear-growth}). Moreover, they are automatically satisfied for functionals with $(p,q)-$growth (see Example \ref{examplepq}).
Note that, traditionally, reference is made to $(p,q)-$growth conditions when the Lagrangian satisfies the ellipticity conditions \eqref{H}$_1$ with $g_1(t)=mt^{p-2}$ and $g_2(t)=Mt^{q-2}$. In the autonomous case, when $1<p<q$, for sufficiently large $|\xi|$, this implies the $(p,q)-$growth condition (see \cite[Lemma 7.2]{EMMP})
\[
c|\xi|^p \le f(\xi) \le C|\xi|^q.
\]
We observe that, if $1=p<q$, $f$  still satisfies hypotheses \eqref{fcoercitiva} and \eqref{ipotesisug2} since, in this case, \eqref{H}$_1$ implies
\[
c|\xi|\log|\xi| \le f(\xi) \le C|\xi|^q
\]
for sufficiently large $|\xi|$.
\\
As for linear growth, and thus in the absence of coercivity, the hypotheses of Theorem 1.1 can still be satisfied, as in Example \ref{linear-growth}. However, this does not hold for all linear growths; for instance, the Lagrangian of the area functional, $f(\xi)=\sqrt{1+|\xi|^2}$, does not satisfy all conditions \eqref{H}.

\medskip The paper is organized as follows. In Section \ref{Section2} we present some examples of admissible choices for the functions $g_1$ and $g_2$ satisfying our structural assumptions {\tt (RGA)}.


 Among them,
 we underline that, in Example \ref{exponential}, we consider an anisotropic exponential growth setting which, to the best of our knowledge, has received little attention in the literature. We can compare this example with those considered in \cite[Theorem 1.4]{DFM21} and \cite[Section 3.2]{MNP26}, observing that the two settings are substantially different. Indeed, in both \cite{DFM21} and \cite{MNP26}, the functionals are isotropic in the gradient and the exponential growth is determined pointwise by the spatial variable $x$ (so that different choices of $x$ correspond to different exponential growth), whereas in our autonomous framework the anisotropy is intrinsic to the growth itself and does not arise from an explicit $x$-dependence.

Section \ref{Preliminary lemma} collects a number of technical preliminary results that will be used throughout the paper. In particular, Proposition \ref{ProposizioneNatale} plays a crucial role in the proof of the a priori estimate, which is carried out in Section \ref{A priori estimates}. 

Finally, in Section \ref{Section5} we complete the proof of the main theorem. We emphasize the role of the coercivity assumption, which is the key ingredient allowing us to pass to the limit in the approximation procedure and preserve the minimality property.

To conclude this introduction we would like to mention that the previous papers \cite{EMMP, mar96, MNP26} on general fast or slow growth are based respectively  on the  increasing and decreasing properties of the functions $g_1$ and $g_2$ while
the main novelty of our approach is that, 
we replace these monotonicity assumptions  with condition \eqref{H}$_{2}$.

Moreover, it is worth to highlight that, as in \cite{EMMP}, $f$ is assumed to be $\mathcal{C}^2$ and to satisfy the  growth assumptions \eqref{H}  only for large values of $\left\vert\xi \right\vert$. Therefore, our work fits into a research direction introduced by \cite{CE86} and motivated by the assumption, in the context of the study of elastic energy, that materials cannot bear excessively high energy densities. For further references on this topic, we refer to \cite{GT}.
\bigskip

\section{Some examples}

\label{Section2}

\bigskip
\begin{example}
\label{examplepq}
{\it (($p,q$)$-$growth)} This is the case when \eqref{H}$_1$ holds for $g_{1}(t)=mt^{p-2}$ and $g_{2}(t)=Mt^{q-2}$, $1\leq p<q$. 
Therefore $t^2g_1(t)=mt^p$, so \eqref{H}$_{2}$ is satisfied.
\\
Moreover, for
every $t\geq t_0$, 
\begin{equation*}
t^4g_2(t)=Mt^{q+2}
\end{equation*}%
and
\[
\int_{t_0}^ts\sqrt{g_1(s)}ds
=\sqrt{m}\int_{t_0}^t s^{p/2}ds=\frac{2\sqrt{m}}{p+2}\,t^{(p+2)/2}-\frac{2\sqrt{m}}{p+2}\,t_0^{(p+2)/2}
\]
then \eqref{H}$_{3}$ holds if $\delta\,\frac{p+2}2=q+2$, that is $\delta=2\frac{q+2}{p+2}$. 
\\
By Lemma 7.2 in \cite{EMMP}, there exists a constant $c>0$ such that for $|\xi |$ large enough
$1+f(\xi)\geq c|\xi|^p$. Therefore for $\alpha=(q+2)/p$ we have
\begin{equation*}
g_{2}(|\xi |)|\xi |^{4}=M|\xi |^{q+2}= M|\xi |^{p\alpha} \leq \frac{M}{c^\alpha}[1+f(\xi )]^{\alpha }
\end{equation*}%
so all the conditions in \eqref{H} are satisfied. 
\\
By Lemma 7.2 in \cite{EMMP}, $f(\xi)\geq \frac m2 |\xi|\log |\xi|$ and $tg_2(t)=mt^{q-1}$, therefore \eqref{fcoercitiva} and \eqref{ipotesisug2} hold.
\\
 Therefore, all the assumptions of Theorem \ref{superlinear}
are satisfied if \eqref{Hpalphadelta} is satisfied, that is 
\begin{equation*}
\alpha <\frac 2n\frac{\delta}{\delta-2} 
\quad\iff\quad
\frac{q+2}{p}<\frac 2n \,\frac{q+2}{p+2}\,\frac{p+2}{q-p}
\quad\iff\quad
\frac{q}{p}<\frac {n+2}n.
\end{equation*}%

\end{example}

\medskip
\begin{example}
\label{linear-growth}
{\it (Linear growth)} Consider  a $\mathcal{C}^2(\R^n)$ function satisfying 
\eqref{supp1}
with
\[
f(\xi)=g(|\xi|)=|\xi|-|\xi|^p, \qquad 0<p<1
\]
out of a fixed ball.
In this case there exist $m,M>0$, such that, for $t$ sufficiently large,
\[
g_1(t)= mt^{p-2}
\quad\text{and}\quad
g_2(t)=Mt^{-1}.
\]
We remark that this situation can be regarded as a  $(p,q)-$growth case with $p<q=1$.
\\
Therefore in \eqref{H} we can take $\delta=\frac{6}{p+2}$ and $\alpha=3$, so that \eqref{Hpalphadelta} is satisfied if
\[\frac 1p<\frac{n}{n-2}.\]
Since the coercivity condition \eqref{fcoercitiva} is not satisfied, the local Lipschitz continuity of the minimizers cannot be guaranteed.
\end{example}

\medskip
\begin{example}
\label{exponential}
{\it (Anistropic exponential growth)} 
\[
f(\xi) = e^{p |\xi|} + e^{q \xi_1}
\]
$0<p\leq q$.
Notice that $f(\xi) =f_p(\xi)+f_q(\xi)$ where $f_p(\xi)= e^{p |\xi|}$ and $f_q(\xi)=e^{q \xi_1}$ .
Therefore the quadratic form associated to the Hessian matrix turns to be $\mathcal{Q}(\lambda)=\mathcal{Q}_p(\lambda)+\mathcal{Q}_q(\lambda)$,
where $\mathcal{Q}_q(\lambda)=q^2 e^{q\xi_1}\lambda_1^2$ and, by (3.3) in \cite{MP06}, for every $\xi$, $|\xi| \geq 1$,
\[
\frac{p}{|\xi|} e^{p|\xi|}|\lambda|^2
\leq\mathcal{Q}_p(\lambda)\leq
p^2 e^{p|\xi|}|\lambda|^2.
\]
Therefore, for every $\xi$, $|\xi| \geq 1$,
\[
\frac{p}{|\xi|} e^{p|\xi}|\lambda|^2
\leq\mathcal{Q}(\lambda)\leq
p^2 e^{p|\xi|}|\lambda|^2+q^2 e^{q\xi_1}\lambda_1^2
\leq 2q^2e^{q|\xi|}|\lambda|^2,
\]
so that \eqref{H}$_1$ holds with $t_0=1$, $g_1(t)=\frac{p}{t} e^{pt}$ and $g_2(t)=2q^2e^{qt}$.
\\
Condition \eqref{H}$_2$  is obviously satisfied. Moreover, \eqref{H}$_3$ and \eqref{H}$_4$ hold if $\alpha>q/p$ and $\delta>2q/p$ respectively. Therefore, there exist $\alpha$ and $\delta$  satisfying \eqref{Hpalphadelta} whenever 
\begin{equation}\label{pq}
    \frac{q}{p}<\frac{n+2}{n}.
\end{equation} 
 In this case, if \eqref{pq} holds, it follows from Theorem~\ref{superlinear} that the minimizers of \eqref{energy integral} are locally Lipschitz.
 \end{example}


\section{Some technical results}

\label{Preliminary lemma}

The proof of the a priori estimate relies on some preliminary and technical result that we collect and prove in this section.

Hypothesis (\ref{H})$_{3}$ can be rewritten, without loss of generality, assuming $t_0=1$. Hence, from now on, we will use the following  equivalent formulation.
\begin{equation}
\label{etichetta}
\left [1+\int_{0}^t (1+ s)\sqrt{g_1(1+s)}\,ds \geq C_1(1+t)^{4/\delta} g_2^{1/\delta}(1+t) \right ], \quad\forall t\geq 0.
\end{equation}
\begin{lemma}
Inequality \eqref{etichetta} turns out to be equivalent to 
\begin{equation}
\label{salamedoca}
1+\int_{0}^t s\sqrt{g_1(1+s)}\,ds \geq C \, (1+t)^{4/\delta} g_2^{1/\delta}(1+t)
\end{equation}
for some constant $C > 0$ and for all $t \geq 0.$
\end{lemma}

\begin{proof}
Obviously \eqref{salamedoca} implies \eqref{etichetta}. To show the opposite inequality, it is sufficient to prove that there exists a constant $K > 0$ such that, for all $t \ge \, 0$
\[
1+\int_{0}^t (1 + s) \sqrt{g_1(1+s)}\,ds \le \, K \left [1+\int_{0}^t s \sqrt{g_1(1+s)}\,ds \right ].
\]
We have
\[
\begin{split}
 1+\int_{0}^t (1 + s) \sqrt{g_1(1+s)}\,ds
&= 1+\tilde{K} +\int_{1}^t \sqrt{g_1(1+s)}\,ds  + \int_{0}^t s \sqrt{g_1(1+s)}\,ds \\
&\le 1 + \tilde{K} + \int_1^t s \cdot \sqrt{g_1(1 + s)}\, ds + \int_{0}^t s \sqrt{g_1(1+s)}\,ds \\
&\le 1 + \tilde{K}  + 2 \int_{0}^t s \sqrt{g_1(1+s)}\,ds \le K \left [1 + \int_{0}^t s \sqrt{g_1(1+s)}\,ds \right ],
\end{split}
\]
where
\[
\tilde{K} := \int_{0}^1 \sqrt{g_1(1+s)}\,ds\qquad K := \max \{(1 + \tilde{K}), 2\}.
\]
\end{proof}

We proceed with the following technical lemma that is a part of \cite[Lemma 3.4]{mar93}. 
\begin{lemma}
\label{lemmaMarcellini}
Let $g(t), h(t)$ two non negative, increasing functions on $[0, + \infty).$ Then
\[
\int_a^b g(t) \, dt \, \cdot \int_a^b h(t) \, dt \le \, (b - a) \, \int_a^b g(t) h(t) \, dt \qquad \forall a,b \in [0, \infty).
\]
\end{lemma}
The following technical lemma instead is a part of \cite[Lemma 2.2]{EMM3}. 
\begin{lemma}
\label{lemma_Marcellini_Advances}
Let $\gamma_0 > 0.$ Then there exists a positive constant $C'$ depending on $\gamma_0$ but independent of $\gamma \ge \, \gamma_0$ and of $t \ge 0$ such that
\[
(1 + t)^{\gamma} \le \, C' \, \gamma^2 \left (1 + \int_0^t (1 + s)^{\gamma - 2} s \, ds \right )
\]
for every $\gamma \in [\gamma_0, + \infty)$ and every $t \in [0, + \infty).$
\end{lemma}

The following result will be important when dealing with the a priori estimate.
\begin{proposition}
\label{ProposizioneNatale}
Assume that $g_1, g_2: [1, + \infty) \rightarrow (0, + \infty)$ are such that
\begin{equation}
\left\{ 
\begin{split}
& t \mapsto t^2 g_1(t) \,\,\, \textnormal{is increasing}\\
&  1 + \int_0^t s \sqrt{g_1(1 + s)} \, ds \ge \, C \, (1 + t)^{4/\delta} g_2(1 + t)^{1/\delta} \\
& \qquad \qquad \textnormal{for some $\delta > 0$ and some $C > 0.$}
\end{split}
\right. \label{HpMichela}
\end{equation}
Then there exists $\bar{C} > 0$ such that $\forall \gamma \ge 0$ and $\forall t \ge 0$
\begin{equation}
\label{tesipropStefania}
1 + \int_0^t (1 + s)^{\gamma} s \, \sqrt{g_1(1 + s)} \, ds \ge \, \bar{C} \, \left [1 + \frac{(1 + t)^{\gamma + 4/\delta}}{{(\gamma+1)^2}} g_2(1 + t)^{1/\delta} \right].
\end{equation}
\end{proposition}
\begin{proof} 
To prove this result, we need to distinguish different cases.
\\[2mm]
$\bullet$ \fbox{{\sc case $t \ge 1,$ $0 \le \gamma \le 1$}}\\[2mm] 
Let us denote
\[
\sigma := \int_0^1 \sqrt{(1 + s)^2 g_1(1 + s)} \, ds > 0
\]
Now we have
\begin{eqnarray}
\int_0^t (1 + s) \sqrt{g_1(1 + s)} \, ds & = & \frac{1}{2} \int_0^t (1 + s) \sqrt{g_1(1 + s)} \, ds + \frac{1}{2}\int_0^t (1 + s) \sqrt{g_1(1 + s)} \, ds \nonumber\\
&\ge & \frac{1}{2}\int_0^1 (1 + s) \sqrt{g_1(1 + s)} \, ds + \frac{1}{2} \int_0^t (1 + s) \sqrt{g_1(1 + s)} \, ds \nonumber\\
&=& \frac{\sigma}{2} + \frac{1}{2} \int_0^t (1 + s) \sqrt{g_1(1 + s)} \, ds \nonumber\\
&\ge & C_3\left (1 + \int_0^t (1 + s) \sqrt{g_1(1 + s)} \, ds \right ) \nonumber \\
&\stackrel{\eqref{HpMichela}}{\ge}& C_4 (1 + t)^{4/\delta} g_2(1 + t)^{1/\delta}, \label{casotmaggioredi1}
\end{eqnarray}
with 
\[
C_3 := \min \left \{ \frac{1}{2}, \frac{\sigma}{2}\right \}.
\]
We notice that \eqref{casotmaggioredi1} actually holds for all $t\ge 1$ and it is independent of $\gamma$. This will be useful in the sequel.

At this point we aim to apply  Lemma \ref{lemmaMarcellini},  with the choice $g(s) := (1 + s)^{\gamma-1} s$ (which is increasing for ${\gamma\ge 0}$) and $h(s) := (1 + s) \sqrt{g_1(1 + s)},$ which is increasing due to \eqref{H}$_2$. 
Therefore we are able to deduce
\begin{eqnarray}
&& 1 + \int_0^t (1 + s)^{\gamma} s \, \sqrt{g_1(1 + s)} \,ds \nonumber\\
&=& 1 + \int_0^t (1 + s)^{\gamma-1} s \, (1 + s) \sqrt{g_1(1 + s)} \,ds \nonumber\\
&\stackrel{\textnormal{Lemma \ref{lemmaMarcellini}}}{\ge}& 1 + \frac{1}{t} \int_0^t (1 + s)^{\gamma-1} s \, ds \, \int_0^t (1 + s) \sqrt{g_1(1 + s)} \, ds \label{ideadiStefania} 
\end{eqnarray}
At this point, being $\gamma\le 1$ we can estimate the last term of the previous chain of inequalities as in the sequel
\begin{eqnarray*}
&& 1 + \frac{1}{t} \int_0^t (1 + s)^{\gamma-1} s \, ds \, \int_0^t (1 + s) \sqrt{g_1(1 + s)} \, ds\\
&\stackrel{\gamma \le 1}{\ge}& 1 + \frac{1}{t} (1+t)^{\gamma-1} \frac{t^2}2 \,  \int_0^t (1 + s) \sqrt{g_1(1 + s)} \, ds \\
&{\stackrel{\eqref{casotmaggioredi1}}{\ge}}& 1 + \frac{t}{2} (1+t)^{\gamma - 1}\, C_4 (1 + t)^{4/\delta} g_2(1 + t)^{1/\delta}  \\
&{=}& 1 + \frac{1}{2} (1+t)^{\gamma-1}(1+t-1)\, C_4 (1 + t)^{4/\delta} g_2(1 + t)^{1/\delta}  \\
& {=}& 1 + \frac{1}{2} (1+t)^{\gamma}\left(1-\frac{1}{1+t}\right)\, C_4 (1 + t)^{4/\delta} g_2(1 + t)^{1/\delta}\\
& {\geq}& 1 + \frac{1}{4} (1+t)^{\gamma}\, C_4 (1 + t)^{4/\delta} g_2(1 + t)^{1/\delta},
\end{eqnarray*}
where in the last line we used the fact that, $\forall t \ge 1$
\[
1 - \frac{1}{1 + t} \ge  \, 1 - \frac{1}{2} = \frac{1}{2}.
\]
This ends the discussion of the case $t \ge 1$ and $0 \le \gamma \le 1.$
\\[2mm]
$\bullet$ \fbox{{\sc case $0 \le t < 1$ and $0 \le \gamma \le \, 1$}}\\[2mm] 
In this case we have
\begin{eqnarray*}
Q(t, \gamma) &:=& \frac{\displaystyle 1 + \int_0^t (1 + s)^{\gamma} s \, \sqrt{g_1(1 + s)} \, ds}{\displaystyle 1 + \frac{(1 + t)^{\gamma + 4/\delta}}{(\gamma+1)^2} g_2(1 + t)^{1/\delta}}\\
&\ge&  \frac{1}{1 + 2^{1 + 4/\delta} \max_{t \in [0,1]} g_2(1+t)^{1/\delta}} =: \bar{C}
\end{eqnarray*}
which yields \eqref{tesipropStefania}.
\\[2mm]
$\bullet$ \fbox{{\sc case $t \ge \, 1$ and $\gamma>1$}}\\[2mm]
Arguing as in \eqref{ideadiStefania} we have
\[
1 + \int_0^t (1 + s)^{\gamma} s \sqrt{g_1(1 + s)} \, ds \ge \,  1 + \frac{1}{t} \int_0^t (1 + s)^{\gamma-1} s \, ds \, \int_0^t (1 + s) \sqrt{g_1(1 + s)} \, ds. 
\]
At this point we aim to apply another time Lemma \ref{lemmaMarcellini}
with the choice $g(s) := (1 + s)^{\gamma-1}$ and $h(s) := s.$ 
Therefore we are able to deduce
\begin{eqnarray*}
&& 1 + \frac{1}{t} \int_0^t (1 + s)^{\gamma-1} s \, ds \, \int_0^t (1 + s) \sqrt{g_1(1 + s)} \, ds \\
&\stackrel{\textnormal{Lemma \ref{lemmaMarcellini}}}{\ge}& 1 + \frac{1}{t^2} \int_0^t (1 + s)^{\gamma-1} \, ds \, \int_0^t s \, ds \, \int_0^t (1 + s) \sqrt{g_1(1 + s)} \, ds \\
&\stackrel{\eqref{casotmaggioredi1}}{\ge}& 1 + \left [\frac{1}{2} \int_0^t (1 + s)^{\gamma-1} \, ds \right ] \, C_4 (1 + t)^{4/\delta} g_2(1 + t)^{1/\delta}  \\
&=& 1 + \frac{1}{2} \left[\frac{(1 + s)^{\gamma}}{\gamma} \right]^t_0 \, C_4 (1 + t)^{4/\delta} g_2(1 + t)^{1/\delta}  \\
&=& 1 + \frac{1}{2\gamma} \left [(1 + t)^{\gamma} - 1 \right] \, C_4 (1 + t)^{4/\delta} g_2(1 + t)^{1/\delta}  \\
&=& 1 + \frac{1}{2} \frac{(1 + t)^{\gamma}}{\gamma} \left [1 - \frac{1}{(1 + t)^{\gamma}} \right] \, C_4 (1 + t)^{4/\delta} g_2(1 + t)^{1/\delta}  \\
&\ge& 1 + \frac{C_4}{4} \frac{(1 + t)^{\gamma}}{\gamma} (1 + t)^{4/\delta} g_2(1 + t)^{1/\delta},
\end{eqnarray*}
where in the last line we used the fact that, $\forall \gamma\ge 1, \forall t \ge 1$
\[
1 - \frac{1}{(1 + t)^{\gamma}} \ge \, 1 - \frac{1}{2^{\gamma}} \ge \, 1 - \frac{1}{2} = \frac{1}{2}.
\]
\\
$\bullet$ \fbox{{\sc case $0 \le t < 1$ and $\gamma > \, 1$}}\\[2mm]
We observe that, by \eqref{HpMichela}$_1$, the inequality $(1+s) \sqrt{g_1(1 + s)} \ge \, g_1(1)$ holds $\forall s \in [0, + \infty)$, therefore
\begin{eqnarray}
&& 1 + \int_0^t (1 + s)^{\gamma}s \, \sqrt{g_1(1 + s)} \, ds \nonumber\\
&= & 1 + \int_0^t (1 + s)^{\gamma-1} s \, (1 + s) \sqrt{g_1(1 + s)} \, ds \\
&\ge & 1 + \sqrt{g_1(1)}\int_0^t (1 + s)^{\gamma-1} s \, ds \nonumber\\
&= & \frac 12+\frac 12 + \sqrt{g_1(1)}\int_0^t (1 + s)^{\gamma-1} s \, ds \nonumber\\
&\ge & \frac 12+m\left(1 + \int_0^t (1 + s)^{\gamma-1} s \, ds \right) \nonumber\\
&\stackrel{\textnormal{Lemma \ref{lemma_Marcellini_Advances}}}{\ge}& \frac{1}{2} + \frac{m}{C'}\frac{1}{(\gamma+1)^2} (1+t)^{(\gamma+1)} \label{t<1gamma>1}
\end{eqnarray}
where we set
\[
m := \min \{\sqrt{g_1(1)}, 1/2\}.
\]
On the other hand, by setting
\[
M := \max_{t \in [0,1]} \left( (1 + t)^{4/\delta - 1} g_2(1 + t)^{1/\delta}\right ) > 0
\]
we have
\begin{eqnarray*}
&& 1 + \int_0^t (1 + s)^{\gamma} s \, \sqrt{g_1(1 + s)} \, ds \\
& \stackrel{\eqref{t<1gamma>1}}{\ge}&  \frac{1}{2} + \frac{m}{C'}\frac{1}{(\gamma+1)^2} (1+t)^{\gamma+1} \frac{M}{M}\\
&=&  \frac{1}{2} + \frac{m}{C' M} \frac{(1 + t)^{\gamma+1}}{(\gamma+1)^2} \,(1 + t)^{4/\delta - 1} g_2(1 + t)^{1/\delta}
\end{eqnarray*}
therefore also in this case \eqref{tesipropStefania} is proved with $\bar{C} := \min \left \{\frac{1}{2}, \frac{m}{C' M} \right \}$ independent of $\gamma.$
\end{proof}

\section{Proof of Theorem \ref{step 1}}

\label{A priori estimates}

The present section is entirely devoted to the proof of the a priori estimate formulated in Theorem \ref{step 1}.
\begin{proof}[Proof of Theorem \protect\ref{step 1}]
Since the local minimizer $u$ is in $W_{\mathrm{loc}}^{1,\infty }(\Omega )$,
it satisfies the following Euler equation: for every open set $\Omega ^{\prime }$
compactly contained in $\Omega $ we have 
\begin{equation*}
\int_{\Omega }\sum_{i=1}^{n}f_{\xi _{i}}(Du)\,\varphi _{x_{i}}\,dx=0\qquad
\forall \varphi \in W_{0}^{1,2}(\Omega ^{\prime }).
\end{equation*}%
Moreover, by the techniques of the difference quotient (see for example \cite%
[Ch.~8, Sect.~8.1]{giusti}), $u\in W_{\mathrm{loc}}^{2,2}(\Omega )$, then
the second variation holds: 
\begin{equation*}
\int_{\Omega }\sum_{i,j=1}^{n}f_{\xi _{i}\xi _{j}}(Du)u_{x_{j}x_{k}}\varphi
_{x_{i}}\,dx=0,\qquad \forall k=1,\dots ,n,\,\,\,\forall \varphi \in
W_{0}^{1,2}(\Omega ^{\prime }).
\end{equation*}%
In the sequel, we denote by $B_{R}$ a generic ball of radius $R$ compactly
contained in $\Omega $ and by $B_{\varrho }$ a ball of radius $\varrho <R$
concentric with $B_{R}$.
\\
For fixed $k=1,\dots ,n$ let $\eta \in \mathcal{C}_{0}^{1}(B_R) $ be equal to 1 in $B_{\rho }$, 
such
that $|D\eta |\leq \,\frac{2}{(R-\rho )},$ and consider $\varphi =\eta
^{2}\,u_{x_{k}}\,\Phi ((|Du|-1)_{+})$ with $\Phi $ non negative, increasing,
locally Lipschitz continuous on $[0,+\infty )$, such that $\Phi (0)=0.$ Here 
$(a)_{+}$ denotes the positive part of $a\in \mathbb{R}$; in the following
we denote for brevity $\Phi ((|Du|-1)_{+})=\Phi (|Du|-1)_{+}.$ Then a.e. in $\Omega $ 
\begin{equation*}
\varphi _{x_{i}}=2\eta \,\eta _{x_{i}}u_{x_{k}}\Phi (|Du|-1)_{+}+\eta
^{2}u_{x_{i}x_{k}}\Phi (|Du|-1)_{+}+\eta ^{2}u_{x_{k}}\Phi ^{\prime
}(|Du|-1)_{+}[(|Du|-1)_{+}]_{x_{i}}.
\end{equation*}%
Proceeding along the lines of \cite{mar96}, we therefore deduce that 
\begin{eqnarray*}
&&\int_{\Omega }2\eta \Phi (|Du|-1)_{+}\sum_{i,j=1}^{n}f_{\xi _{i}\xi
_{j}}(Du)u_{x_{j}x_{k}}\eta _{x_{i}}u_{x_{k}}\,dx \\
&&+\int_{\Omega }\eta ^{2}\Phi (|Du|-1)_{+}\sum_{i,j=1}^{n}f_{\xi _{i}\xi
_{j}}(Du)u_{x_{j}x_{k}}u_{x_{i}x_{k}}\,dx \\
&&+\int_{\Omega }\eta ^{2}\Phi ^{\prime }(|Du|-1)_{+}\sum_{i,j=1}^{n}f_{\xi
_{i}\xi _{j}}(Du)u_{x_{j}x_{k}}u_{x_{k}}[(|Du-1|)_{+}]_{x_{i}}\,dx=0.
\end{eqnarray*}%
We estimate the first integral in the previous equation by using the
Cauchy-Schwarz inequality and the Young inequality, so that 
\begin{equation*}
\begin{split}
& \left\vert \int_{\Omega }2\eta \Phi (|Du|-1)_{+}\sum_{i,j=1}^{n}f_{\xi
_{i}\xi _{j}}(Du)u_{x_{j}x_{k}}\eta _{x_{i}}u_{x_{k}}\,dx\right\vert \\
& \quad \leq \int_{\Omega }2\Phi (|Du|-1)_{+}\left( \eta
^{2}\sum_{i,j=1}^{n}f_{\xi _{i}\xi
_{j}}(Du)u_{x_{i}x_{k}}u_{x_{j}x_{k}}\right) ^{\frac{1}{2}}\left(
\sum_{i,j=1}^{n}f_{\xi _{i}\xi _{j}}(Du)\eta _{x_{i}}u_{x_{k}}\eta
_{x_{j}}u_{x_{k}}\right) ^{\frac{1}{2}}\,dx \\
& \quad \leq \frac{1}{2}\int_{\Omega }\eta ^{2}\Phi
(|Du|-1)_{+} \sum_{i,j=1}^{n}f_{\xi _{i}\xi
_{j}}(Du)u_{x_{i}x_{k}}u_{x_{j}x_{k}}\,dx \\
& \qquad +2\int_{\Omega }\Phi (|Du|-1)_{+}\sum_{i,j=1}^{n}f_{\xi _{i}\xi
_{j}}(Du)\eta _{x_{i}}u_{x_{k}}\eta _{x_{j}}u_{x_{k}}\,dx.
\end{split}%
\end{equation*}%
Therefore we deduce 
\begin{equation}
\label{(*)}
\begin{split}
\frac{1}{2}\int_{\Omega }\eta ^{2}\Phi (|Du|-1)_{+}& \sum_{i,j=1}^{n}f_{\xi
_{i}\xi _{j}}(Du)u_{x_{i}x_{k}}u_{x_{j}x_{k}}\,dx \\
 + & \int_{\Omega }\eta ^{2}\Phi ^{\prime }(|Du|-1)_{+}\sum_{i,j=1}^{n}f_{\xi
_{i}\xi _{j}}(Du)u_{x_jx_k}u_{x_{k}}[(|Du|-1)_{+}]_{x_{i}}\,dx \\
 \leq & \,2\int_{\Omega }\Phi (|Du|-1)_{+}\sum_{i,j=1}^{n}f_{\xi _{i}\xi
_{j}}(Du)\eta _{x_{i}}u_{x_{k}}\eta _{x_{j}}u_{x_{k}}\,dx.
\end{split}%
\end{equation}%
Since
\begin{equation*}
\lbrack (|Du|-1)_{+}]_{x_{j}}=%
\begin{cases}
(|Du|)_{x_{j}}=\frac{1}{|Du|}\sum_{k}u_{x_{j}x_{k}}u_{x_{k}} & \text{ if $%
|Du|>1$,} \\ 
0 & \text{ if $|Du|\leq 1$,}%
\end{cases}%
\end{equation*}%
a.e.~in $\Omega$, $\forall j=1,\dots n$, we obtain 
\begin{equation*}
\sum_{k=1}^{n}\sum_{i,j=1}^{n}f_{\xi _{i}\xi
_{j}}(Du)u_{x_{j}x_{k}}u_{x_{k}}[(|Du|-1)_{+}]_{x_{i}}=|Du|%
\sum_{i,j=1}^{n}f_{\xi _{i}\xi
_{j}}(Du)[(|Du-1|)_{+}]_{x_{j}}[(|Du|-1)_{+}]_{x_{i}}.
\end{equation*}%
Summing up in \eqref{(*)} for $k = 1, \dots, n$ 
\begin{equation*}
\begin{split}
\int_{\Omega }\eta ^{2}& \Phi (|Du|-1)_{+}\sum_{k,i,j=1}^{n}f_{\xi _{i}\xi
_{j}}(Du)u_{x_{j}x_{k}}u_{x_{i}x_{k}}\,dx \\
& +\int_{\Omega }\eta ^{2}|Du|\Phi ^{\prime
}(|Du|-1)_{+}\sum_{i,j=1}^{n}f_{\xi _{i}\xi
_{j}}(Du)[(|Du-1|)_{+}]_{x_{j}}[(|Du|-1)_{+}]_{x_{i}}\,dx \\
& \leq 4\int_{\Omega }\Phi (|Du|-1)_{+}\sum_{k,i,j=1}^{n}f_{\xi _{i}\xi
_{j}}(Du)\eta _{x_{i}}u_{x_{k}}\eta _{x_{j}}u_{x_{k}}\,dx.
\end{split}%
\end{equation*}%
Using the inequality $|D(|Du|-1)_{+}|^{2}\leq \,|D^{2}u|^{2}$ and the
ellipticity condition in \eqref{H}$_{1}$ we obtain 
\begin{equation}
\begin{split}
& \int_{\Omega } \eta ^{2}[\Phi (|Du|-1)_{+}]{
g_{1}(1+(|Du|-1)_{+}))}\,|D(|Du|-1)_{+}|^{2}\,dx\\ 
\leq & \int_{\Omega } \eta ^{2}[\Phi (|Du|-1)_{+} +|Du|\Phi ^{\prime }(|Du|-1)_{+}]{
g_{1}(1+(|Du|-1)_{+}))}\,|D(|Du|-1)_{+})^{2}\,dx \\
 = & \int_{\Omega }\eta ^{2}[\Phi (|Du|-1)_{+}+|Du|\Phi ^{\prime }(|Du|-1)_{+}]%
{g_{1}(|Du|)}\,|D(|Du|-1)_{+}|^{2}\,dx \\
\leq & \,4\,\int_{\Omega }|D\eta |^{2}\,\Phi
(|Du|-1)_{+}\,g_{2}(|Du|)\,|Du|^{2}\,dx \\
= &\,4\,\int_{\Omega }|D\eta |^{2}\,\Phi
(|Du|-1)_{+}\,g_{2}(1+(|Du|-1)_{+}))\,|Du|^{2}\,dx.
\end{split}
\label{(25)Jota}
\end{equation}%
Let us define 
\begin{equation}
G(t)=1+\int_{0}^{t}\sqrt{\Phi (s)\,g_{1}(1+s)}\,ds\qquad \forall t\geq \,0,
\label{defG}
\end{equation}%
where $\Phi (t)=(1+t)^{2\gamma} t^{2}$, $\gamma\geq 0$.
Since  $t\mapsto
t^{2\gamma+2} g_{1}(t)$ is increasing,
\begin{equation*}
\begin{split}
G(t)& =1+\int_{0}^{t}\sqrt{\Phi (s)g_{1}(1+s)}\,ds =1+\int_{0}^{t}\sqrt{(1+s)^{2\gamma}s^2g_{1}(1+s)}\,ds\\
&\leq 1+\int_{0}^{t}\sqrt{(1+s)^{2\gamma+2}g_{1}(1+s)}\,ds\leq 1+\sqrt{(1+t)^{2\gamma+2}g_{1}(1+t)}\int_{0}^{t} 1 \,ds\\
&= 1+\sqrt{(1+t)^{2\gamma+2}t^2g_{1}(1+t)}
\end{split}
\end{equation*}
hence, using the fact that $g_1 \le \, g_2,$ 
\[
[G(t)]^{2}\leq 2 \left[ 1+(1+t)^{2\gamma+2}t^2g_{2}(1+t)\right],
\]
 we deduce
\begin{equation}
\label{stimadiGconildeumeno1+}
[G((|Du|-1)_{+})]^{2}\leq 2 \left[1 +(1+(|Du|-1)_{+})^{2\gamma+2} ((|Du|-1)_{+})^2g_{2}(1+(|Du|-1)_{+})\right].
\end{equation}

On the other hand 
\begin{equation*}
\begin{split}
|D& (\eta (G((|Du|-1)_{+})|^{2} \\
& \leq \,2\,|D\eta |^{2}[G((|Du|-1)_{+})]^{2}+2\eta ^{2}[G^{\prime
}((|Du|-1)_{+})]^{2}\,|D((|Du|-1)_{+})|^{2} \\
&\quad +2\,\eta
^{2}\,(1+(|Du|-1)_{+})^{2\gamma}(|Du|-1)_{+}^2g_{1}(1+(|Du|-1)_{+})\,|D(1+(|Du|-1)_{+})|^{2}.
\end{split}%
\end{equation*}%
Since $(|Du(x)|-1)_{+}=0$ when $|Du(x)|\leq 1$%
, by \eqref{(25)Jota} we get 
\begin{equation}
\begin{split}
& \int_{\Omega }|D(\eta \,G((|Du|-1)_{+}))|^{2}\,dx \\
& \le \, c_1 \int_{\Omega }|D\eta |^{2}\,\left[1+(1+(|Du|-1)_{+})^{2\gamma+2}(|Du|-1)_{+}^2\,g_{2}(1+(|Du|-1)_{+})\right]\,dx.
\end{split}
\label{Sobolev3}
\end{equation}%
At this point, by Proposition \ref{ProposizioneNatale}, we have
\[
G(t) = 1 + \int_0^t (1 + s)^{{\gamma}} s \, \sqrt{g_1(1 + s)} \, ds \ge \, \bar{C} \, \left [1 + \frac{(1 + t)^{{{\gamma} + 2/\delta}}}{(\gamma + 1)^2}[(1+t)^2g_2(1 + t)]^{1/\delta}\right].
\]
Since $(1+t)^2g_2(1+t)\geq (1+t)^2g_1(1+t)\geq g_1(1)>0$ for every $t\geq0$ and $\delta<2^*$, we have that 
\[[(1+t)^2g_2(1+t)]^{1/\delta} \ge \, g_2(1)^{\frac 1\delta -\frac{1}{2^*}} [(1+t)^2g_2(1+t)]^{1/2^*}.\]
Therefore there exists a constant $c_2>0$ such that
\begin{eqnarray*}
G((|Du|-1)_{+}) &=&1+\int_{0}^{(|Du|-1)_{+}}(1+s)^{{\gamma}}s
\sqrt{g_{1}(1+s)}\,ds \\
&\geq & \frac{c_2}{(\gamma + 1)^2}\left[1+g_{2}(1+(|Du|-1)_{+})^{\frac{1}{2^{\ast }}}
(1+(|Du|-1)_{+})^{\gamma + 2/\delta+2/2^*}\right]
\end{eqnarray*}%
By the Sobolev inequality, there exists a constant $c_{S}$ such that 
\begin{equation}
\left\{ \int_{\Omega }[\eta \,G((|Du|-1)_{+})]^{2^{\ast }}\,dx\right\}
^{2/2^{\ast }}\leq \,c_{S}\,\int_{\Omega }|D(\eta (G(|Du|-1)_{+}))|^{2}\,dx
\label{Sobolev1}
\end{equation}%
where $2^{\ast }$ is equal to $\frac{2n}{n-2}$ if $n>2$ or it is an arbitrary number greater than 2 if $n=2$, thus by \eqref{Sobolev3} and \eqref{Sobolev1} we obtain that there exists $c_3>0$ such
that, for all $\gamma \geq 0$, 
\[
\begin{split}
& \left\{ \int_{\Omega }\eta ^{2^{\ast }}(1+(1+(|Du|-1)_{+})^{2^{\ast} ({\gamma} +
2/\delta)+2}\,g_{2}(1+(|Du|-1)_{+}))\,dx\right\} ^{\frac{2%
}{2^{\ast }}} \\
& \qquad \leq \, c_3\,(\gamma + 1)^4 \int_{\Omega }|D\eta |^{2}\,[1+(1+(|Du|-1)_{+})^{2 {(\gamma + 2)}}\,g_{2}(1+(|Du|-1)_{+})]\,dx 
\end{split}
\]
\\
By the property of the cut-off function $\eta$,
we finally obtain
\begin{equation}
\begin{split}
\label{deliriuniversali}
& \left\{ \int_{B_{\rho}} (1+(1+(|Du|-1)_{+})^{2^{\ast} ({\gamma} + 2/\delta)+2}\,g_{2}(1+(|Du|-1)_{+}))\,dx\right\} ^{\frac{2%
}{2^{\ast }}} \\
& \qquad \leq \,4 c_3\,\frac{(\gamma + 1)^4}{(R - \rho)^2}\int_{B_R} \,[1+(1+(|Du|-1)_{+})^{2 {(\gamma + 2)}}\,g_{2}(1+(|Du|-1)_{+})]\,dx, 
\end{split}%
\end{equation}%
for every $\gamma\geq 0$, $0<\rho<R$, $\overline{B}_\rho\subset\overline{B}_R\subset\Omega$.
\\
The iteration
process follows now the arguments contained in \cite{mar96}; for clarity, we recall the main steps here.
\\
Fix $\bar\rho, \bar R$, $0<\bar\rho<\bar R$, such that $\overline{B}_{\bar \rho}\subset\overline{B}_{\bar R}\subset\Omega$ and set $\bar r :=\frac{\bar\rho+\bar R}{2}$.
We define by induction a sequence $\gamma_k$ in the following way
\[
{\gamma_1 = 0}, \qquad 2^* \left ({\gamma_{k}} + \frac{2}{\delta} \right ) + 2 = 2 \,{(\gamma_{k+1} + 2)} \qquad \forall k \in \mathbb{N}, \,\,\, k > 1.
\]
It is not difficult to prove the following representation formula for $\gamma_k$
\begin{equation}
\gamma_k = {
 \frac{2^*2/\delta -2}{2^* - 2} \, \left [ \left (\frac{2^*}{2} \right )^{k-1} - 1  \right ]}, \qquad \forall k \in \mathbb{N}, \,\,\,\forall k \ge \, 1. \label{gammak}   
\end{equation}
We observe that 
$\frac{2^*2}{\delta}-2>0 $ if and only if $
\delta <  2^* $ .
Therefore, for every $k \in \mathbb{N}, k \ge 1$ we obtain
\begin{displaymath}
2(\gamma_{k+1}+2)=2\frac{\frac{2^*2}{\delta}-2}{2^*-2}\left(\frac{2^*}{2}\right)^k-2\frac{\frac{2^*2}{\delta}-2}{2^*-2}+4. 
\end{displaymath}
Moreover
\begin{displaymath}
4-2\frac{\frac{2^*2}{\delta}-2}{2^*-2}\geq 0
\end{displaymath}
since $\delta>2>\frac{2^*}{2^*-1}$. Hence, we conclude that
\begin{equation}
\label{numero}
2(\gamma_{k+1}+2)\geq 2\frac{\frac{2^*2}{\delta}-2}{2^*-2}\left(\frac{2^*}{2}\right)^k, 
\qquad \forall k \in \mathbb{N}, \,\,\, \forall k \ge \, 1.
\end{equation}
For  all $k \in \mathbb{N}$, $k \ge 1$ we rewrite \eqref{deliriuniversali} with $R = r_{k-1}$ and $\rho = r_k,$ where
\[
r_k = \bar\rho + \frac{(\bar{r} - \bar\rho)}{2^k}
\]
and where we put $\gamma$ equal to $\gamma_k.$ 
We thus have  
\begin{equation}
\begin{split}
\label{deliriuniversali2}
& \left\{ \int_{B_{r_k}} (1+(1+(|Du|-1)_{+})^{2 (\gamma_{k+1} + 2) }\,g_{2}(1+(|Du|-1)_{+}))\,dx\right\} ^{\frac{2%
}{2^{\ast }}} \\
& \qquad \leq \,C_k \,\int_{B_{r_{k-1}}} [1+(1+(|Du|-1)_{+})^{2(\gamma_k
+2)}\,g_{2}(1+(|Du|-1)_{+})]\,dx,
\end{split}%
\end{equation}%
where
\[
C_k = c_3 \, \frac{(\gamma_k + 1)^4 \, 2^{2k}}{(\bar r - \bar\rho)^2}.
\]
Now let us fix $i \in \mathbb{N};$ by iterating  \eqref{deliriuniversali2} for $k = 1, \dots, i$ we obtain 
\begin{equation}
\begin{split}
\label{deliriuniversali1}
& \left\{ \int_{B_{r_i}} (1+(1+(|Du|-1)_{+})^{2 (\gamma_{i+1} + 2) }\,g_{2}(1+(|Du|-1)_{+}))\,dx\right\}^{\left (\frac{2%
}{2^{\ast }}\right )^k} \\
& \leq \,C_i^{\left (\frac{2%
}{2^{\ast }}\right )^{k-1}} \, C_{i-1}^{\left (\frac{2}{2^{\ast }}\right )^{k-2}}  \, \dots \, C_{i - k +1} \\
&\qquad \times \int_{B_{r_{i-k}}} [1+(1+(|Du|-1)_{+})^{2 (\gamma_{i-k+1} + 2)}\,g_{2}(1+(|Du|-1)_{+})]\,dx 
\end{split}%
\end{equation}%
For $k = i,$ $r_{i-k} = r_0 = \bar r$, we obtain  
\begin{equation}
\begin{split}
& \left\{ \int_{B_{r_i}}\left[ 1+(1+(|Du|-1)_{+})^{2(\gamma_{i + 1} + 2)}\,g_{2}(1+(|Du|-1)_{+})\right] \,dx\right\} ^{\left (\frac{2}{%
2^{\ast }}\right )^i} \\
& \qquad \leq \,\prod_{k = 1}^{i} C_k^{\left (\frac{2}{2^*} \right )^{k-1}}\,\int_{B_{\bar r}}\left[ 1+(1+(|Du|-1)_+)^{4} g_{2}(1+(|Du|-1)_{+})%
\right] \,dx.
\end{split}
\label{MP1}
\end{equation}%
By \eqref{gammak} it follows
\begin{equation}
\begin{split}
\gamma_k  &  \le \,  \frac{\frac{2^*2}{\delta}-2}{2^*-2}\left (\frac{2^*}{2} \right )^{k-1}
\end{split}
\label{ipassaggio-meglio}
\end{equation}
 and
\[
\prod_{k = 1}^{i} C_k^{\left (\frac{2}{2^*} \right )^{k-1}} \le \, \prod_{k=1}^{\infty} \left [c_3 \frac{(\gamma_k +1)^4 2^{2k}}{(\bar r - \bar \rho)^2} \right ]^{\left ( \frac{2}{2^*}\right )^{k-1}} \le \, \frac{c_4}{(\bar r - \bar \rho)^{\frac{2 \, 2^*}{2^* - 2}}}.
\]
So we obtain
\begin{equation*}
\begin{split}
& \left[ \int_{B_{{{r_i}}}}(1+(|Du|-1)_{+})^{2(
\gamma_{i+1} + 2) }g_{2}(1+(|Du|-1)_{+})\,dx\right] ^{\left( \frac{2}{2^{\ast }}\right)
^{i}} \\
& \qquad \leq \,{\frac{c_4}{(\bar r -\bar \rho)^{\frac{2\,2^{\ast }}{%
2^{\ast }-2}}}\int_{B_{\bar r}}\left[1 + 
(1+(|Du|-1)_{+})^4g_{2}(1+(|Du|-1)_{+})\right] \,dx.}
\end{split}%
\end{equation*}%
Recalling that,  for every $t \ge 1$, $g_2(t) t^2 \ge \, g_1(1) > 0$ 
and \eqref{numero}, we have
\[
\begin{split}
  (1+&(|Du|-1)_{+})^{2\frac{\frac{2^*2}{\delta}-2}{2^*-2}\left(\frac{2^*}{2}\right)^i - 2} \\
\le &
(1+(|Du|-1)_{+})^{2(\gamma_{i+1} + 1) }\\
\le &
\frac 1{g_1(1)}(1+(|Du|-1)_{+})^{2(\gamma_{i+1} + 2) }
g_{2}(1+(|Du|-1)_{+}) .
\end{split}
\]
Finally,  passing to the limit we obtain that there exists a positive constant $c_4=c_4(g_1(1))$ such that
\begin{equation}
\label{gradiente-infinito}
\begin{split}
&\sup \left\{|Du(x)|:\, x \in B_{\bar \rho}\right\}^{2\frac{\frac{2^*2}{\delta}-2}{2^*-2}} =\lim_{i\rightarrow +\infty }\left[ \int_{B_{{\bar \rho}%
}}(1+(|Du|-1)_{+})^{{2\frac{\frac{2^*2}{\delta}-2}{2^*-2}\left(\frac{2^*}{2}\right)^i - 2}} \,dx \right] ^{\left( 
\frac{2}{2^{\ast }}\right) ^{i}} \\
& \le \lim_{i\rightarrow +\infty }\left[ \int_{B_{{\bar \rho}%
}}(1+(|Du|-1)_{+})^{2(
\gamma_{i+1} + 2) }g_{2}(1+(|Du|-1)_{+})\,dx\right] ^{\left( 
\frac{2}{2^{\ast }}\right) ^{i}} \\
& \leq \, \frac{c_5}{(\bar r - \bar \rho)^{\frac{2\,2^{\ast }}{2^{\ast
}-2}}}\int_{B_{\bar r}}(1+(|Du|-1)_{+})^{{4}}g_{2}(1+(|Du|-1)_{+})\,dx.
\end{split}
\end{equation}
In order to conclude the proof, let us consider the function $G:[0,+\infty)\to \R$  defined in \eqref{defG} with $\gamma=0$
\begin{equation*}
    G(t)=1+\int_0^t{s} \sqrt{g_1(1+s)}ds.
\end{equation*}
Putting together \eqref{Sobolev3} and \eqref{Sobolev1} with $\gamma =0$, we get
\begin{equation}
 \label{stima-Sobolev}
  \left[\int_{\Omega}\left|\eta \,G(|Du|-1)_+\right|^{2^*}\,dx\right]^{\frac2{2^*}}\leq \, c_1 \int_\Omega |D\eta|^2 g_2(1+(|Du|-1)_{+})(1 + (|Du|-1)_+)^{{4}} dx.
\end{equation}
Setting $\nu := \frac{2^*}{\delta}$, by \eqref{salamedoca}, we have
\[
[G(t)]^{2^*}=\left[\left(1+\int_0^ts\sqrt{g_1(1+s)\,ds}\right)^{\frac{2^*}\nu}\right]^\nu
\ge \,{\bar{C}}^{\frac{2^*}\nu}\left[(1+t)^4g_2(1+t)\right]^\nu
\]
so, from \eqref{stima-Sobolev}, there exists a constant $c_6>0$ depending on $n$ and $\delta$ such that
\[
\begin{split}
 &\left[\int_\Omega \eta^{2^*}\left[g_2(1+(|Du|-1)_{+})(1 + (|Du|-1)_+)^{{4}}\right]^{\nu} dx\right]^{2/2^*} \\
& \qquad \le c_6\, \int_\Omega |D\eta|^2 g_2(1+(|Du|-1)_{+})(1 + (|Du|-1)_+)^{{4}} dx.    
\end{split}
\]
If we set
\begin{equation}
\label{definizionediV}
    V(x):=(1+(|Du|-1)_{+})^{{4}} g_2(1+(|Du|-1)_{+}),
\end{equation}
then the previous inequality entails
\begin{equation*}
    \left (\int_{B_{ \rho}} V(x)^{\nu} dx\right )^{2/2^*}\le \frac{c_6}{(R-\rho)^2}\int_{B_{R}}V(x)dx
\end{equation*}
 Let $\beta >\frac{2^*}{2}$, then by H\"older inequality of exponents $\beta$ and $\frac{\beta}{\beta - 1}$, we obtain
\begin{eqnarray}
   \left (\int_{B_\rho}V(x)^{\nu}  dx \right )^{2/2^*} &\le & \frac{c_6}{(R-\rho)^2}\int_{B_{R}}V(x)^{\frac{\nu}{\beta}} \, V(x)^{1 - \frac{\nu}{\beta}}dx \nonumber\\
   &\le &
   \frac{c_6}{(R-\rho)^2}\left(\int_{B_R}V(x)^{\nu}  dx\right)^{\frac{1}{\beta}}\left(\int_{B_R}V(x)^{\frac{\beta-\nu}{\beta-1}}  dx\right)^{\frac{\beta-1}{\beta}}. \label{(46)}
\end{eqnarray}
for every $0<\rho<R$,
$\overline{B}_\rho\subset\overline{B}_R\subset\Omega$.
\\
For  every $k \in \mathbb{N},$ we consider \eqref{(46)} with $R =r_{k}$ and $\rho = r_{k-1},$ with
\[
r_k= \bar R - \frac{\bar R - \bar r}{2^k},
\]
namely
\begin{equation}
\label{(46)iterato}
   \left (\int_{B_{r_{k-1}}}V(x)^{\nu}  dx \right )^{2/2^*} \le\frac{c_6 \, 4^k}{(\bar R-\bar r)^2}\left(\int_{B_{r_k}}V(x)^{\nu}  dx\right)^{\frac{1}{\beta}}\left(\int_{B_{r_k}}V(x)^{\frac{\beta-\nu}{\beta-1}}  dx\right)^{\frac{\beta-1}{\beta}}.
\end{equation}
Let us set
\[
A_k = \int_{B_{r_k}} V(x)^{\nu} \, dx \qquad  \qquad B_k = \int_{B_{r_k}} V(x)^{\frac{\beta - \nu}{\beta - 1}} \, dx.
\]
Then \eqref{(46)iterato} reads
\[
A_{k-1} \le \, \left (c_6 \, \frac{4^k}{(\bar R - \bar r)^2} A_k^{\frac{1}{\beta}} B_k^{\frac{\beta - 1}{\beta}} \right )^{2^*/2}.
\]
By induction on $k$ we deduce
\begin{equation}
\label{induzionei}
A_0 \le \, \frac{\tilde{C}_k}{(\bar R - \bar r)^{2 \sum_{i = 1}^{{k}} \frac{1}{\beta^{i - 1}} \left (\frac{2^*}{2} \right )^i}} \, B_k^{\frac{\beta - 1}{\beta} \sum_{i = 1}^k \frac{1}{\beta^{i - 1}} \left (\frac{2^*}{2} \right )^i} \, A_k^{\frac{1}{\beta^{k}} \left (\frac{2^*}{2} \right )^k}
\end{equation}
where
\[
\tilde{C}_k := c_6^{\sum_{i = 1}^k \frac{1}{\beta^{i - 1}} \left (\frac{2^*}{2} \right )^i}  \, 4^{\sum_{i= 1}^{k} \frac{i}{\beta^{i - 1}} \left (\frac{2^*}{2} \right )^i} .
\]
Now we observe that, since $\beta>2^*/2$,
\begin{eqnarray*}
&& \sum_{i = 1}^k \frac{1}{\beta^{i - 1}} \left (\frac{2^*}{2} \right )^i \le \, \frac{2^*}{2}  \sum_{i = 0}^{\infty} \left (\frac{1}{\beta} \frac{2^*}{2} \right )^i  = \frac{2^*}{2}\, \frac{1}{1 - \frac{2^*}{2 \beta}} =  \frac{2^* \beta}{2 \beta - 2^*}\\[3mm] 
&& \sum_{i = 1}^k \frac{i}{\beta^{i - 1}} \left (\frac{2^*}{2} \right )^i \le \, \frac{2^*}{2} \sum_{i = 0}^{\infty} (i + 1) \left (\frac{1}{\beta} \frac{2^*}{2} \right )^i  < + \infty 
\end{eqnarray*}
thus, letting $k \rightarrow + \infty$, we finally deduce
\[
A_0 = \int_{B_{\bar r}} V(x)^{\nu} \, dx \le \,  \frac{\tilde{C}}{(\bar R - \bar r)^{2 \frac{2^* \beta}{2 \beta - 2^*}}} \left ( \int_{B_{\bar R}} V(x)^{\frac{\beta - \nu}{\beta - 1}} \, dx\right )^{\frac{2^* (\beta - 1)}{2 \beta - 2^*}}
\]
having set
\[
{c_7} = \lim_{k \rightarrow + \infty} \tilde{C}_k.
\]
Let $\beta$ be such that 
\begin{equation}
\label{betanu}
\frac{\beta-\nu}{\beta-1}=\frac{1}{\alpha}
\end{equation}
this is possible since, by \eqref{betanu}
\[
\beta = \frac{\alpha \nu - 1}{\alpha - 1} > \frac{2^*}{2} \quad \Leftrightarrow \quad \alpha < \frac{\delta(2^*-2)}{2^*(\delta - 2)}
\]
which is nothing but \eqref{Hpalphadelta}.
At this point, using \eqref{H}$_4$ in the previous inequality to get
\begin{equation*}
    \int_{B_{\bar r}}V(x)^\nu  dx\le \frac{c_8}{(\bar R - \bar r)^{2 \frac{2^* \beta}{2 \beta - 2^*}}}  \left(\int_{B_{\bar R}} (1+f(Du)) \,  dx\right)^{2^*\frac{\beta-1}{2\beta-2^*}},
\end{equation*}
being $c_8 := c_7 \, C_2$. Using once more H\"older inequality, this time with exponents $\nu$ and $\frac{\nu}{\nu - 1}$ we obtain
\begin{align*}
 & \int_{B_{\bar r}} (1+(|Du|-1)_{+})^{{4}} g_2(1+(|Du|-1)_{+}) \, dx \\
 = & \int_{B_{\bar r}}V(x) \, dx\\
 \le & |B_{\bar r}|^{1-\frac{1}{\nu}}\left(\int_{B_{\bar r}}V(x)^\nu dx\right)^{\frac{1}{\nu}}\\
 \le & \frac{|B_{\bar r}|^{1-\frac{1}{\nu}} c_8^{\frac 1\nu}}{(\bar R - \bar r)^{2 \frac{2^* \beta}{(2 \beta - 2^*) \nu}}} \,\left(\int_{B_{\bar R}} (1+f(Du)) \,    dx\right)^{2^*\frac{\beta-1}{(2\beta-2^*)\nu}}\\
 = & |B_{\bar r}|^{1-\frac{1}{\nu}} c_8^{\frac 1\nu} \,\left(\frac{1}{(\bar R - \bar r)^{\frac{2\beta}{\beta-1}}}\int_{B_{\bar R}} (1+f(Du)) \,    dx\right)^{2^*\frac{\beta-1}{(2\beta-2^*)\nu}} \\
 = & c_9\,\left(\frac{1}{(\bar R - \bar r)^{2\frac{\alpha2^*-\delta}{\alpha(2^*-\delta)}}}\int_{B_{\bar R}} (1+f(Du)) \,    dx\right)^{\frac{\alpha(2^*-\delta)}{22^*\frac{\alpha}{\delta}-2^*\alpha+2^*-2}},
\end{align*}
being 
\[
c_9 := |B_{\bar r}|^{1-\frac{1}{\nu}} c_8^{\frac 1\nu} = |B_{\bar r}|^{\frac{2^* - \delta}{2^*}} c_8^{\frac{\delta}{2^*}}.
\]
To conclude, we observe that \eqref{gradiente-infinito} entails
\[
\begin{split}
&\Vert 1+(|Du|-1)_{+}\Vert _{L^{\infty }(B_{\bar \rho})}\\
& \leq \frac{c_{10}}{(\bar r -\bar\rho)^{\frac{2^*\delta}{2^*2-2\delta}}}
\left (\int_{B_{\bar r}}(1+(|Du|-1)_{+})^{{4}}g_{2}(1+(|Du|-1)_{+})\,dx \right )^{\frac\delta 2\frac{2^*-2}{2^*2-2\delta}}.
\end{split}%
\]
Therefore, for every $\bar\rho, \bar R$,  $0 <\bar \rho < \bar R$, there exists a constant $C>0$ such that 
\begin{align*}
&||Du||_{L^\infty(B_{\bar\rho}\,,\R^n)} \le C \,\left(\frac{1}{(\bar R - \bar\rho)^{\frac{\Theta_1}{\Theta_2}}}\int_{B_{\bar R}} (1+f(Du)) \,    dx\right)^{\Theta_2}.
\end{align*}
with 
\begin{eqnarray*}
\Theta_1 & := & \displaystyle \frac{2^*\delta}{2^*2-2\delta} + 2\frac{\alpha2^*-\delta}{\alpha(2^*-\delta)} \, \frac{\alpha(2^*-\delta)}{22^*\frac{\alpha}{\delta}-2^*\alpha+2^*-2} \, \frac\delta 2\frac{2^*-2}{2^*2-2\delta} \\
& = &\displaystyle \frac{\delta}{2} \,
\frac{2\alpha \, 2^* + \delta \, (2^* - 2)}
{2\alpha \, 2^* - \alpha\delta \, 2^* + \delta \, (2^* - 2)}
\end{eqnarray*}
\[
\Theta_2 = \frac\delta 2 \, \frac{2^*-2}{2^*2-2\delta} \, \frac{\alpha(2^*-\delta)}{22^*\frac{\alpha}{\delta}-2^*\alpha+2^*-2} = \frac{\alpha \delta^2 (2^* - 2)}
{4\left(2\alpha 2^* - \alpha \delta 2^* + \delta(2^* - 2)\right)} 
\]
therefore, referring to the statement of Theorem \ref{step 1}, we set
\begin{eqnarray*}
&& \theta_1 := \frac{\Theta_1}{\Theta_2} = \frac{\delta}{2} \,
\frac{2\alpha \, 2^* + \delta \, (2^* - 2)}
{2\alpha \, 2^* - \alpha\delta \, 2^* + \delta \, (2^* - 2)} \, \frac
{4\left(2\alpha 2^* - \alpha \delta 2^* + \delta(2^* - 2)\right)}{\alpha \delta^2 (2^* - 2)} = \frac{4 \, 2^*}{\delta (2^* - 2)} + \frac{2}{\alpha} 
\end{eqnarray*}
and
\[
\theta_2 := \Theta_2 = \frac{\alpha \delta^2 (2^* - 2)}
{4\left(2\alpha 2^* - \alpha \delta 2^* + \delta(2^* - 2)\right)}.
\]
We observe that 
$\theta_1 > 0;$ moreover 
\begin{eqnarray*}
\theta_2 > 1/2 &\Leftrightarrow & \alpha > \frac{2 \delta (2^* - 2)}{\delta^2 (2^* - 2) + 2^* \, 2 (\delta - 2)} 
\end{eqnarray*}
and we observe that
\begin{eqnarray*}
\frac{2 \delta (2^* - 2)}{\delta^2 (2^* - 2) + 2^* \, 2 (\delta - 2)} < 1 & \Leftrightarrow & (2^* - 2) \delta (2 - \delta) < 2 \, 2^* (\delta - 2)
\end{eqnarray*}
which is obviously true as long as $\delta > 2.$ We remark that these bounds are in agreement with classical results, see for instance \cite{mar96}, where the left hand side is taken with power 2.

\end{proof}



\section{Proof of Theorem \ref{superlinear} }

\label{Section5}

\begin{proof}[Proof of Theorem \protect\ref{superlinear}]

We can argue as in the proof of Theorem 2.1 in \cite{EMMP}. Here we summarize the main steps not only for the sake of completeness but also to underline the fact that we can use our assumption \eqref{ipotesisug2} instead of the more restrictive hypothesis (2.1)$_2$ in \cite{EMMP}.

We start by considering the sequence of  Lagrangians $\{\tilde f_k\}_k$, uniformly approximating $f$, defined at page 1028 of \cite{EMMP}. We notice that, at the already mentioned page, it is underlined that $\tilde f_k$ is $\mathcal{C}^2(\R^n)$, it is uniformly elliptic on compact sets and, moreover, it satisfies a condition of the type
\[
 g_{1}\left( \left\vert \xi \right\vert \right) \left\vert \lambda
\right\vert ^{2}\leq \sum_{i,j=1}^{n}\tilde f_{k_{\xi _{i}\xi_{j}}}\left( \xi \right)
\lambda _{i}\lambda _{j}\leq  2 g_{2}\left( \left\vert \xi \right\vert \right)
\left\vert \lambda \right\vert ^{2}\,,\;\;\;\forall \;\lambda ,\xi \in 
\mathbb{R}^{n},\;\left\vert \xi \right\vert {\geq }t_{0},
\]
provided that $t\mapsto tg_2(t)$ is bounded from below by a strictly positive constant in the set $[t_0,+\infty)$. Our assumption \eqref{ipotesisug2} easily implies that this is true for a suitable $t_0$.

We consider a local minimizer $u$ and a sequence $v_k\in\mathcal{C}^\infty(\Omega')$, $\overline{B}_R\subset\Omega'\Subset\Omega$ of functions obtained mollifying $u$. Thanks to our assumption \eqref{fcoercitiva}, the uniform convergence and the uniform ellipticity from below of $\tilde f_k$ in $\R^n$, it follows that, for any $k\in\N$ there exists a unique minimizer $u_k$ of the functional
\[
\int_{B_R} \tilde f_k(D v)\,dx\qquad v\in v_k+W^{1,1}_0(B_R).
\]

Since the boundary datum $v_k$ belongs to $ \mathcal{C}^{\infty}(B_R)$, then it satisfies the Bounded Slope Condition and  we can apply the semiclassical theory of the Calculus of Variations (see \cite{giusti}, Chapter 1). Observing also that, for this class of functionals the Lavrentiev phenomenon does not occurs, we can conclude that $u_k\in v_k+W^{1,\infty}_0(B_R)$
and we can apply our Theorem \ref{step 1} to obtain an estimate, uniform with respect to $k$, of the Lipschitz constant of $u_k$ on $B_\rho$.

The coercivity assumption \eqref{fcoercitiva} and again the properties of the sequences $\tilde f_k$, $v_k$ allow us to say that the sequence $u_k$ converges to a minimizer of \eqref{energy integral} in $u+W^{1,1}_0(B_R)$, weakly  in  $W^{1,1}(B_R)$  and weakly$^*$ in $W^{1,\infty}(B_r)$, for every $0<r<R$. Hence we conclude that there exists a local minimizer that is locally Lipschitz. Following the same argument used in \cite{EPT25} we obtain also that any local minimizer is locally Lipschitz. 
\end{proof}

\end{document}